\documentclass[12pt,english]{amsart}
\usepackage{amsfonts,amsmath,amsthm,amscd,amssymb,latexsym,amsbsy,pb-diagram,caption,url}
\usepackage[space, noadjust]{cite}
\usepackage[mathscr]{eucal}
\usepackage[T2A]{fontenc}
\usepackage[cp1251]{inputenc}
\usepackage[english]{babel}
\usepackage{tikz} 
\usetikzlibrary{positioning}
\usepackage{geometry}
\newcounter{saveeqnu}

\newtheorem{theorem}{Theorem}[section]
\newtheorem{corollary}[theorem]{Corollary}
\newtheorem{lemma}[theorem]{Lemma}
\newtheorem{proposition}[theorem]{Proposition}

\theoremstyle{definition}
\newtheorem{definition}[theorem]{Definition}
\theoremstyle{remark}
\newtheorem{remark}[theorem]{Remark}
\newtheorem{example}[theorem]{Example}
\numberwithin{equation}{section}
\DeclareMathOperator{\diam}{diam}
\DeclareMathOperator{\card}{card}

\newcommand{\Sp}[1]{\operatorname{Sp}(#1)}
\newcommand{\ve}{\varepsilon}
\def\we{\mathrel{\stackrel{\rm w}=}}

\newcommand{\NN}{\mathbb{N}}
\newcommand{\RR}{\mathbb{R}}

\makeatletter
\@namedef{subjclassname@2020}{%
  \textup{2020} Mathematics Subject Classification}
\makeatother

\begin{document}

\title
[On quasisymmetric mappings in semimetric spaces]
{On quasisymmetric mappings in semimetric spaces}

\author{Evgeniy Petrov}

\address{Institute of Applied Mathematics and Mechanics of the NAS of Ukraine,
Dobrovolskogo str. 1, 84100 Slovyansk, Ukraine}

\email{eugeniy.petrov@gmail.com}

\author{Ruslan Salimov}

\address{Institute of Mathematics of the NAS of Ukraine,
Tereschenkivska str. 3,  01024 Kiev, Ukraine}

\email{ruslan.salimov1@gmail.com}

\subjclass[2020]{54E25, 54C25}

\keywords{quasisymmetric mapping, semimetric space, triangle function, weak similarity}

\begin{abstract}
The class of quasisymmetric mappings on the real axis was first introduced by A.~Beurling and L.~V.~Ahlfors in 1956. In 1980 P.~Tukia and J.~V\"{a}is\"{a}l\"{a} considered these mappings between general metric spaces. In our paper we generalize the concept of quasisymmetric mappings to the case of general semimetric spaces and study some properties of these mappings. In particular, conditions under which quasisymmetric mappings preserve triangle functions, Ptolemy's inequality and the relation ``to lie between'' are found.  Considering quasisymmetric mappings between semimetric spaces with different triangle functions we have found a new estimation for the ratio of diameters of two subsets, which are images of two bounded subsets. This result generalizes the well-known Tukia-V\"{a}is\"{a}l\"{a} inequality. Moreover, we study connections between quasisymmetric mappings and weak similarities which are a special class of mappings between semimetric spaces.
\end{abstract}

\maketitle

\section{Introduction}

The fundamental concept of metric space was introduced by M. Fr\'{e}chet~\cite{Fr06} in 1906. Fr\'{e}chet called the discovered spaces ``classes (D)'' (from the word ``distance''). F. Hausdorff~\cite{Ha14} introduced the term “metric space” in 1914 considering these spaces as a special case of infinite topological spaces.

Let $X$ be a nonempty set. Recall that a mapping  $d\colon X\times X\to \mathbb{R}^+$, $\mathbb{R}^+=[0,\infty)$ is a \emph{metric} if for all $x,y,z \in X$ the following axioms hold:
\begin{itemize}
  \item [(i)] $(d(x,y)=0)\Leftrightarrow (x=y)$,
  \item [(ii)] $d(x,y)=d(y,x)$,
  \item [(iii)] $d(x,y)\leqslant d(x,z)+d(z,y)$.
\end{itemize}
The pair $(X,d)$ is called a \emph{metric space}. If only axioms (i) and (ii) hold then $d$ is called a \emph{semimetric}. A pair $(X,d)$, where  $d$  is a semimetric on $X$, is called a \emph{semimetric space}.
Such spaces were first examined by Fr\'{e}chet in~\cite{Fr06}, where he called them ``classes (E)''. Later these spaces attracted the attention of many mathematicians ~\cite{Ch17,Ni27,Wi31,Fr37}.

Note that in the literature a different terminology is used.
Sometimes a semimetric space is called a distance space~\cite{G16}; a semimetric also can be called a dissimilarity~\cite{DF98}. In~\cite[Chapter 10]{KV88}, the topological space $(X, \tau_d)$ with the topology generated by $d$ is called a symmetric space, whereas a semimetric space means a symmetric space in which all open balls are neighbourhoods.
In our paper we inherit the terminology from Wilson's pioneering paper~\cite{Wi31}, as well as it is adopted in a well-known Blumenthal's work~\cite[p.~7]{Bl53} and many recent papers, e.g.,~\cite{DP13, BP17, DH17, KS15}.

In 2017 M. Bessenyei and Z. P\'ales~\cite{BP17} introduced a definition of a triangle function $\Phi \colon \overline{\RR}_+^2\to \overline{\RR}^+$ for a semimetric $d$.
We use this definition in a slightly different form restricting the domain and the range of $\Phi$ by ${\RR}_+^2$ and ${\RR}^+$, respectively.

\begin{definition}\label{d2}
Consider a semimetric space $(X, d)$. We say that $\Phi \colon {\RR}^+\times{\RR}^+ \to {\RR}^+$ is a \emph{triangle function} for $d$ if $\Phi$ is symmetric and monotone increasing in both of its arguments, satisfies $\Phi(0,0)=0$ and, for all $x, y, z \in X$, the generalized triangle inequality
$$
d(x,y)\leqslant \Phi(d(x,z), d(y,z))
$$
holds.
\end{definition}
Obviously, metric spaces are semimetric spaces with the triangle function $\Phi(u, v) = u + v$.
In~\cite{BP17} those semimetric spaces whose so-called basic triangle functions are continuous at the origin were considered. These spaces were termed regular. It was shown that the topology of a regular semimetric space is Hausdorff, that a convergent sequence in a regular semimetric space has a unique limit and possesses the Cauchy property, etc. See also~\cite{JT20,VH17,CJT18} for some new results in this direction.

Quasisymmetric mappings play an important role in the theory of quasiconformal mappings. These mappings on the real line were first introduced by A. Beurling and L. V. Ahlfors~\cite{BA56}, who extended a quasisymmetric self-mapping of the real line to a quasiconformal self-mapping of the upper half-plane.
The necessary and sufficient condition on the function obtained in this case formed the basis for the definition of quasisymmetric functions introduced by J.A. Kelingos in 1966 and studied by him in~\cite{K66}. In the plane case, the first developments in the theory of quasisymmetric mappings belong to H. Renggli~\cite{Re71}, who considered mappings satisfying the boundedness condition for the distortion of a triangle.
In 1980 the Finnish mathematicians P. Tukia and J. V\"{a}is\"{a}l\"{a}~\cite{TV80} noticed that the definition given by H. Renggli can be extended to the case of general metric spaces, which allowed them to separate the class of $\eta$-quasisymmetric mappings.

Since the appearance of the seminal paper~\cite{TV80} these mappings  are intensively studied by many mathematicians all over the world (see, e.g.~\cite{BHW16, Va88, LVZ18, BK00, Ib10, Va81, LH16, As09, AB19}). It is also worth mentioning the following remarkable results. In 1998, J. Heinonen and P. Koskela~\cite{HK98} showed that the concepts of quasiconformality and quasisymmetry are quantitatively equivalent in a large class of metric spaces including Euclidean spaces. Later J. V\"{a}is\"{a}l\"{a} proved the quantitative equivalence between free quasiconformality and quasisymmetry of homeomorphisms between Banach spaces, see~\cite[Theorem 7.15]{Va99}.

In our paper we extend the definition from~\cite{TV80} to the case of general semimetric spaces.
\begin{definition}\label{d12}
Let $(X,d)$, $(Y,\rho)$ be semimetric spaces. We shall say that a mapping $f\colon X\to Y$ is $\eta$-\emph{quasisymmetric} if there is a homeomorphism $\eta\colon [0, \infty)\to [0,\infty)$ so that
\begin{equation}\label{e0}
d(x,a)\leqslant t d(x,b)  \, \text{ implies } \, \rho(f(x),f(a))\leqslant \eta(t)\rho(f(x),f(b))
\end{equation}
for all triples $a,b,x$ of points in $X$ and for all $t>0$.
\end{definition}

Note that in the definition from~\cite{TV80} the mapping $f\colon X \to Y$ was supposed to be an embedding, where $X$ and $Y$ are metric spaces.
The paper is organized as follows. In Section~\ref{genprop2} we discuss some topological properties of semimetric spaces and some properties of $\eta$-quasisymmetric mappings in general semimetric spaces. In Theorem~\ref{p24} of Section~\ref{s2} we found conditions under which the image $f(X)$ of a semimetric space $X$ with the triangle function $\Phi_1$ under $\eta$-quasisymmetric $f$ is a semimetric space with another triangle function $\Phi_2$. As corollaries of Theorem~\ref{p24} we obtain conditions under which $f$ preserves b-metricity and ultrametricity.
Conditions under which $f$ preserves the Ptolemy's inequality and the metric betweenness are found.
In Section~\ref{s4} considering quasisymmetric mappings between semimetric spaces with different triangle functions we have found a new estimation for the ratio of diameters of two subsets which are images of two bounded subsets. This result generalizes the well-known Tukia-V\"{a}is\"{a}l\"{a} inequality, see Theorem 2.5 in~\cite{TV80}.  In Section~\ref{s5} we study connections between $\eta$-quasisymmetric mappings and weak similarities which are a special class of mappings between semimetric spaces, see~\cite{DP13}. To be exact, conditions under which $\eta$-quasisymmetric mappings are weak similarities and conditions under which weak similarities are $\eta$-quasisymmetric mappings are found.

\section{Some properties of semimetric spaces and quasisymmetric mappings}\label{genprop2}

Note that in semimetric spaces the standard notions from the theory of metric spaces like convergence, open (closed) ball, etc. can be introduced in the usual way.

\begin{example}\label{ex21}
Let us show that in semimetric spaces a sequence can have more than one limit. Consider the subset $X=\{-1\}\cup\{0\}\cup \{\frac{1}{n}\,|\, n\in \NN\}$ of the real line. Define a semimetric space $(X,d)$ as follows: consider that all distances between points of the set $X$ are Euclidian except the distances between $-1$ and $\frac{1}{n}$ which we redefine to be $d(-1,\frac{1}{n})=\frac{1}{n}$. Hence, the sequence $ (\frac{1}{n})_{n\in \NN}$ has two limits $-1$ and $0$.
\end{example}

In the general case a semimetric $d\colon X\times X\to \RR^+$ generates a natural topology on $X$. Let $(X,d)$ be a semimetric space. Consider a topology $\tau_d$ on $X$ defined in the following way: a set $A\subseteq X$ is open in $(X,\tau_d)$ if and only if for every $a\in A$ there is $r>0$ such that $B(a,r)\subseteq A$, where $B(a,r)=\{x\in X\colon d(a,x)<r\}$.

\begin{remark}\label{r22}
If a semimetric space $(X,d)$ is defined as in Example~\ref{ex21}, then a simple verification shows that every open ball $B(x,r)\subseteq X$ is open in $(X,\tau_d)$, but the closed ball
$$
\overline{B}\left(0,\frac{1}{2}\right)=
\left\{ x\in X\colon d(0,x)\leqslant \frac{1}{2} \right\}
$$
is not a closed subset of $(X,\tau_d)$.
\end{remark}
More generally we have the following.
\begin{proposition}\label{p23}
Let $(X,d)$ be a semimetric space. If every closed ball of $(X,d)$ is a closed subset of the topological space $(X,\tau_d)$, then, for all points $x,y \in X$ and each sequence $(x_n)_{n\in \NN}$, the double equality
\begin{equation}\label{e(1)}
  \lim\limits_{n\to \infty}d(x_n,x)=0=
  \lim\limits_{n\to \infty}d(x_n,y)
\end{equation}
implies that $x=y$.
\end{proposition}
\begin{proof}
Let $x$ and $y$  be different points of $X$ satisfying~(\ref{e(1)}) for a sequence $(x_n)_{n\in \NN}\subseteq X$ and let $r:=d(x,y)$.  We evidently have $r>0$  and
$$
y\notin
\overline{B}\left(x,\frac{r}{2}\right)=
\left\{z\in X\colon d(x,z) \leqslant \frac{r}{2} \right\}.
$$
It suffices to show that the closed ball $\overline{B}(x,\frac{r}{2})$ cannot be a closed subset of the topological space $(X,\tau_d)$. Towards this, suppose that $\overline{B}(x,\frac{r}{2})$ is closed in $(X,\tau_d)$. Then $X\setminus \overline{B}(x,\frac{r}{2})$ is open in  $(X,\tau_d)$ and $y \in X\setminus \overline{B}(x,\frac{r}{2})$. Hence, by definition of $\tau_d$, there is $r^* \in (0,\infty)$ such that the open ball $B(y,r^*)$  is a subset $X\setminus \overline{B}(x,\frac{r}{2})$. In particular, the inclusion $B(y,r^*)\subseteq X\setminus \overline{B}(x,\frac{r}{2})$ implies
$$
B\left(x,\frac{r}{2}\right)\cap B(y,r^*)=\varnothing,
$$
contrary to~(\ref{e(1)}).
\end{proof}

It was shown in Theorem~5.6 of~\cite{CJT18} that for every unbounded connected metric space $(X,\rho)$, there exists a semimetric $d$ on $X$ such that $d$ and $\rho$ are Lipschitz equivalent, but, for all $x\in X$ and $r>0$, we have $B(x,r)\notin\tau_d$ and $X\setminus \overline{B}(x,r) \notin \tau_d$. Consequently, the converse to Proposition~\ref{p23} is false.

Let $A$ be a subset of a topological space $X$. Recall that \emph{interior} of $A$ ($\operatorname{Int}(A)$) is the union of all open subsets of $X$ contained in $A$. The following lemma is a reformulation of a known characterization of the first-countable topologies $\tau_d$ generated by semimetrics $d$ (see, for example, p.~277 in~\cite{HNV04}).

\begin{lemma}\label{l24}
Let $(X,d)$ be a semimetric space. Then $\tau_d$ is first-countable if and only if $x\in \operatorname{Int}(B(x,r))$ holds for every $x\in X$ and every $r\in (0,\infty)$.
\end{lemma}
\begin{proposition}\label{p25}
Let $(X,d)$ be a semimetric space. Suppose that $(X,\tau_d)$ is first countable. Then the following statements are equivalent:
\begin{itemize}
  \item [(i)] For all points $x,y\in X$ and each sequence $(x_n)_{n\in \NN}\subseteq X$, double equality~(\ref{e(1)}) implies that $x=y$.
  \item [(ii)] The topological space $(X,\tau_d)$ is Hausdorff.
\end{itemize}
\end{proposition}
\begin{proof}
(i)$\Rightarrow$(ii). Let (i) hold. Let us consider two different points $x$ and $y$ in $(X,\tau_d)$. We must show that there exist some sets $U_x$, $U_y \in \tau_d$ such that
\begin{equation}\label{e(2)}
  x\in U_x, \quad y\in U_y, \quad U_x\cap U_y=\varnothing.
\end{equation}
Let $(r_n)_{n\in \NN}$ be a sequence in $(0,\infty)$ so that $\lim_{n\to\infty} r_n=0$. By Lemma~\ref{l24}, we obtain $x\in \operatorname{Int}(B(x,r_n))$ and $y\in \operatorname{Int}(B(y,r_n)) \in \tau_d$ for all $n\in \NN$. If there is $n_0\in \NN$ such that
$$
\operatorname{Int}(B(x,r_{n_0}))\cap \operatorname{Int}(B(y,r_{n_0}))=\varnothing,
$$
then~(\ref{e(2)}) holds with $U_x = \operatorname{Int}(B(x,r_{n_0}))$ and $U_y = \operatorname{Int}(B(y,r_{n_0}))$. Otherwise, for each $n\in \NN$ we can find $x_n\in X$ such that
\begin{equation}\label{e(3)}
x_n\in \operatorname{Int}(B(x,r_n))\cap \operatorname{Int}(B(y,r_n)).
\end{equation}
The equality $\lim_{n\to \infty} r_n=0$ and~(\ref{e(3)}) imply that
$$
  \lim\limits_{n\to \infty}d(x_n,x)=0=
  \lim\limits_{n\to \infty}d(x_n,y).
$$
Consequently, by Statement (i), $x=y$ holds, contrary to $x\neq y$.

(ii)$\Rightarrow$(i)
Let $(X,\tau_d)$ be a Hausdorff topological space. Let us consider $x,y\in X$ and $(x_n)_{n\in \NN}\subseteq X$ satisfying~(\ref{e(1)}). Since $(X,\tau_d)$ is Hausdorff, there are disjoint $U_x, U_y \in \tau_d$ such that $x\in U_x$ and $y\in U_y$. It follows from the definition of $\tau_d$ that $B(x,r) \subset U_x$ and $B(y,r)\subset U_y$ for some $r>0$. Since $U_x$ and $U_y$ are disjoint,
\begin{equation}\label{e(4)}
B(x,r)\cap B(y,r)=\varnothing
\end{equation}
holds. To complete the proof, it suffices to note that~(\ref{e(1)}) implies
$$
x_n\in B(x,r_n)\cap B(y,r_n)\subseteq B(x,r)\cap B(y,r)
$$
whenever $n\in \NN$ is large enough, contrary to~(\ref{e(4)}).
\end{proof}
\begin{remark}
There are many different examples of semimetric spaces, where a sequence of points can have more than one limit (see, for example,~\cite{Tu21} for a collection of related examples). Propositions~\ref{p23} and~\ref{p25}, which give conditions for the uniqueness of the limits, look new, although, the uniqueness conditions for limits of sequences in semimetric spaces were studied already in Wilson’s paper~\cite{Wi31}.
\end{remark}

The topology $\tau_d$ in the general case may not be suitable for use (an open ball $B(a, r)$ is not necessarily an open set). Therefore in the following proposition we consider a class of semimetric spaces $(X,d)$ in which the semimetric $d$ has a natural property of continuity: for any pair of points $a, b \in X$ and any $\ve>0$ there exists $\delta>0$ such that for all $x\in B(a,\delta)$ and $y\in B(b,\delta)$ the inequality $|d(x,y)-d(a,b)|<\ve$ holds. In this case the semimetric $d$ is called continuous. In particular, any metric $d$ has this property. The topology generated by a continuous semimetric is always a Hausdorff topology. The system of sets of the form $B(a,r)$ gives the neighborhood base at the point $a$ in this topology.

Recall that in general topology an \emph{embedding} is a homeomorphism onto its image.

The following proposition generalizes Theorem 2.21 from~\cite{TV80} to the case when $X$, and $Y$ are semimetric spaces. The proof is analogous to the proof from~\cite{TV80} but we reproduce it here with slight modifications for the convenience of the reader.

\begin{proposition}\label{p21}
Let $(X,d)$, $(Y,\rho)$ be semimetric spaces with continuous $d$ and $\rho$, $\eta\colon [0, \infty)\to [0,\infty)$ be a homeomorphism, and $f\colon X\to Y$ be a mapping such that implication~(\ref{e0}) holds for for all triples $a,b,x$ of points in $X$ and for all $t>0$. Then $f$ is either constant or an $\eta$-quasisymmetric embedding.
\end{proposition}
\begin{proof}
Let $x_0\in X$ and $\ve>0$. Fix $b\in X$, $b \neq x_0$. Choose $t>0$ such that $\eta(t)\rho(f(b),f(x_0))<\ve$. Then $\rho(f(x),f(x_0))<\ve$ for $x\in B(x_0, td(b,x_0))$. Thus, $f$ is continuous.

Suppose that $f$ is not constant. If $f(x)=f(y)$ for some $x\neq y$ and if $z\neq x$, then setteing $t=d(z,x)/d(y,x)$, we obtain $\rho(f(z),f(x))\leqslant \eta(t)\rho(f(y),f(x))=0$. Hence, $f$ is injective.

It remains to prove that $f^{-1}\colon f(X) \to X$ is continuous at an arbitrary point $f(x_0)$. Suppose first that $x_0$ is not isolated in $X$. Let us show that for every $\ve$ there exists $\delta>0$ such that $f^{-1}(B(f(x_0),\delta))\subseteq B(x_0,\ve)$ or equivalently $B(f(x_0),\delta)\subseteq f(B(x_0,\ve))$.
Suppose it is not so. Hence, for some $\ve>0$ and for all $\delta>0$ there exists $x\notin B(x_0,\ve)$ such that $f(x)
\in B(f(x_0),\delta)$. Since $x_0$ is not isolated, there exists  $x'$  such that $d(x_0,x')<\ve$. Since $d(x_0,x)\geqslant \ve$, we have $d(x_0,x')<d(x_0,x)$ and $\rho(f(x_0), f(x'))\leqslant \eta(1)\rho(f(x_0), f(x))$. Hence, for $\delta=\rho(f(x_0), f(x'))/\eta(1)$ we have $f(x)\notin B(f(x_0), \delta)$ for all $x$ such that $d(x_0,x)\geqslant \ve$.

Suppose that $x_0$ is isolated in $X$. We must show that $f(x_0)$ is isolated in $fX$. If this is not true, there is a sequence of points $x_j\in X$, $x_j\neq x_0$, such that $f(x_j)\to f(x_0)$. Choose $t>0$ such that $d(x_1,x_0)\leqslant t d(x_j,x_0)$ for all $j\geqslant 1$. Then $\rho(f(x_1),f(x_0))\leqslant \eta(t)\rho(f(x_j),f(x_0))\to 0$, which gives a contradiction.
\end{proof}

Note that according to Definition~\ref{d12} the mapping $f$ is not supposed to be continuous. Only formal implication~(\ref{e0}) is essential. All the following results of this work are focused mainly on the study of distance properties of semimetric spaces determined by the existence of a quasisymmetric mapping between them.
Thus, in what follows we do not consider any topologies in semimetic spaces.

\begin{proposition}
Let $(X,d)$, $(Y,\rho)$ be semimetric spaces, $f\colon X\to Y$ be an $\eta$-quasisymmetric mapping, and let $\operatorname{card}(X)\geqslant 2$. Then
\begin{equation}\label{e22}
  \eta(t)\eta\left(\frac{1}{t}\right)\geqslant 1
\end{equation}
for all $t=d(x,a)/d(x,b)$, where $x, a, b \in X$, $x\neq a$, $x\neq b$.

Moreover, $\eta(1)\geqslant 1$.
\end{proposition}

\begin{proof}
Let $x, a, b \in X$, $x\neq a$, $x\neq b$, and let $t=d(x,a)/d(x,b)$. Hence, $1 / t=d(x,b)/d(x,a)$. By~(\ref{e0}) we have
$$
\rho(f(x),f(a))\leqslant \eta(t) \rho(f(x),f(b))
$$
and
$$
\rho(f(x),f(b))\leqslant \eta\left(\frac{1}{t}\right) \rho(f(x),f(a)).
$$
Hence,
\begin{equation}\label{e}
\frac{1}{\eta(\frac{1}{t})} \rho(f(x),f(b))\leqslant \rho(f(x),f(a))\leqslant \eta(t) \rho(f(x),f(b))
\end{equation}
which gives~(\ref{e22}).

Taking $a=b$ we easily get the inequality $\eta(1)\geqslant 1$.
\end{proof}

Recall that a mapping $f$ from a semimetric space $(X, d)$ to a semimetric space $(Y, \rho)$ is a \emph{similarity} if there exists $\lambda>0$ such that
$$
\rho(f(x),f(y))= \lambda d(x,y)
$$
for all $x$, $y \in X$.

\begin{proposition}\label{p248}
Let $(X,d)$, $(Y,\rho)$ be semimetric spaces and let $f\colon X\to Y$ be an $\eta$-quasisymmetric mapping with $\eta(t)=t^{\alpha}$, $\alpha>0$. Then
\begin{equation}\label{e25}
\rho(f(x),f(y))=\lambda (d(x,y))^{\alpha}
\end{equation}
for some $\lambda > 0$. In particular, if $\alpha=1$, then $f$ is a similarity.
\end{proposition}

\begin{proof}
Consider~(\ref{e}) with $\eta(t)=t^{\alpha}$. Hence,
$$
  \rho(f(x),f(a))= t^{\alpha} \rho(f(x),f(b)).
$$
Since $t=d(x,a)/d(x,b)$ we get
$$
\frac{\rho(f(x),f(a))}{(d(x,a))^{\alpha}}=\frac{\rho(f(x),f(b))}{(d(x,b))^{\alpha}}.
$$
Suppose that $a\neq b$ and there exists $c\neq x,a,b$.  Analogously, we can show that
$$
\frac{\rho(f(c),f(b))}{(d(c,b))^{\alpha}}=\frac{\rho(f(x),f(b))}{(d(x,b))^{\alpha}},
$$
i.e.,
$$
\frac{\rho(f(x),f(a))}{(d(x,a))^{\alpha}}=\frac{\rho(f(c),f(b))}{(d(c,b))^{\alpha}}
$$
for all different $x,a,b,c$, which implies~(\ref{e25}).
\end{proof}

\begin{proposition}\label{p34}
Let $(X,d)$, $(Y,\rho)$ be semimetric spaces and let $f\colon X \to Y$ be a mapping satisfying the double inequality
\begin{equation}\label{e30}
\varphi_1( d(x,y) )\leqslant \rho(f(x),f(y))\leqslant \varphi_2  (d(x,y)),
\end{equation}
where $\varphi_1, \varphi_2 \colon [0, \infty) \to [0, \infty)$ are such that  $\varphi_1(0)=\varphi_2(0)=0$, $\varphi_1(t)\leqslant\varphi_2(t)$ for all $t\in \RR^+$ and let $\varphi_1$ and $\varphi_2$ have the following properties:
\begin{itemize}
  \item [(i)] $\varphi_2(t) \text{ is nondecreasing,}$
  \item [(ii)] $\varphi_2(uv)\leqslant C \varphi_2(u)\varphi_2(v)\text{ for some } C>0,$
  \item [(iii)] $\varphi_2(t)\leqslant K\varphi_1(t) \text{ for some } K \geqslant 1$,
  \item [(iv)] $\varphi_1 \colon [0, \infty) \to [0, \infty)$ is a homeomorphism.
\end{itemize}
Then $f$ is an $\eta$-quasisymmetric mapping with $\eta(t)=CK^2\varphi_1(t)$.
\end{proposition}
\begin{proof}
Let $x,a,b \in X$ and $d(x,a)\leqslant td(x,b)$, $t> 0$. Using consecutively the right inequality in~(\ref{e30}), the previous inequality, (i), (ii), (iii), and the left inequality in~(\ref{e30}), we get
\begin{multline}\label{e35}
\rho(f(x),f(a))\leqslant \varphi_2(d(x,a))\leqslant \varphi_2(td(x,b)) \\ \leqslant C\varphi_2(t)\varphi_2(d(x,b)) \leqslant CK^2\varphi_1(t)\varphi_1(d(x,b))\leqslant CK^2\varphi_1(t)\rho(f(x),f(b)).
\end{multline}
The left inequality in~(\ref{e30}) and condition (iv) imply that $f$ is nonconstant. Hence, by Definition~\ref{d12} it follows from~(\ref{e35}) that $f$ is an $\eta$-quasisymmetric mapping with $\eta(t)=CK^2\varphi_1(t)$.
\end{proof}

\begin{example}\label{e3}
Let $\varphi_1(t)=C_1t^{\alpha}$, $\varphi_2(t)=C_2t^{\alpha}$, $0<C_1\leqslant C_2$. Then, taking $C=1/C_2$, $K=C_2/C_1$ we have that $f$ is an $\eta$-quasisymmetric mapping with $\eta(t)=\frac{C_2}{C_1}t^{\alpha}$.
\end{example}

\begin{example}
Analogously, if $\varphi_1(t)=C_1\max\{t^{\alpha}, t^{\frac{1}{\alpha}}\}$,
$\varphi_2(t)=C_2\max\{t^{\alpha}, t^{\frac{1}{\alpha}}\}$,
$0<C_1\leqslant C_2$, then $f$ is an $\eta$-quasisymmetric mapping with $\eta(t)=\frac{C_2}{C_1}\max\{t^{\alpha}, t^{\frac{1}{\alpha}}\}$.
\end{example}

\begin{remark}
Condition (ii) is known as $\Delta'$-condition, see~\cite[p.~43]{KR58}. Some another examples of functions satisfying this condition are the following: $\varphi(t)=|t|^{\alpha}(|\ln|t|+1|)$, $\alpha>1$; $\varphi(t)=(1+|t|)\ln(1+|t|)-|t|$.
\end{remark}

The well-known concept of bi-Lipschitz mappings can be easily generalized to the case of semimetric spaces.
\begin{definition}\label{d31}
Let $(X,d)$, $(Y,\rho)$ be semimetric spaces. A mapping $f\colon X \to Y$ is called \emph{$L$-bi-Lipschitz} if there exists $L\geqslant 1$ such that the relation
\begin{equation}\label{e31}
\frac{1}{L} d(x,y)\leqslant \rho(f(x),f(y))\leqslant L d(x,y).
\end{equation}
holds for all $x,y \in X$.
\end{definition}

Setting in Example~\ref{e3} $\alpha =1$, $C_1=\frac{1}{L}$, $C_2=L$, $L\geqslant 1$
we get as a corollary the following assertion.
\begin{corollary}
Every $L$-bi-Lipschitz mapping is an $\eta$-quasisymmetric mapping with $\eta(t)=L^2t$.
\end{corollary}

\section{Properties of quasisymmetric mappings to preserve \\ the structures of spaces}\label{s2}

\subsection{Property to preserve triangle inequalities}

Metric-preserving functions were in detail studied by many  mathematicians, see, e.g.,~\cite{D96,DM13}.  Recall that a function $f\colon \RR^1\to \RR^1$ preserves metric if the composition $f\circ d$ is a metric on $X$ for any metric space $(X,d)$. In this paper we understand the property to preserve metric (or equivalently triangle inequality) in another way.  In Theorem~\ref{p24} we describe $\eta$-quasisymmetric mappings $f$ for which the image $f(X)$ of a semimetric space $X$ with the triangle function $\Phi_1$ is a semimetric space with another triangle function $\Phi_2$. As a corollaries we obtain conditions under which $f$ preserves b-metricity and ultrametricity.

For the definition of triangle function see Definition~\ref{d2}.

\begin{theorem}\label{p24}
Let $(X,d)$ be a semimetric space with the triangle function $\Phi_1$, $(Y,\rho)$ be a semimetric space and let $f\colon X\to Y$ be a surjective $\eta$-quasisymmetric mapping.
Suppose that the following conditions hold for $\Phi_1$ and for some function $\Phi_2\colon \RR^+\times \RR^+ \to \RR^+$:
\begin{itemize}
  \item [(i)]  $\Phi_2$ is symmetric, monotone increasing in both of its arguments and satisfies $\Phi_2(0,0)=0$,
  \item [(ii)] $\lambda \Phi_1(x,y) \leqslant  \Phi_1(\lambda x,\lambda y)$   and $\Phi_2(\lambda x,\lambda y) \leqslant \lambda \Phi_2(x,y)$ for all $\lambda>0$,
  \item [(iii)]  For all $t_1, t_2\in \RR_+\setminus \{0\}$ the inequality
\begin{equation}\label{e27}
1\leqslant \Phi_1\left(\frac{1}{t_1},\frac{1}{t_2}\right)
\, \text{ implies } \,
1\leqslant \Phi_2\left(\frac{1}{\eta(t_1)},\frac{1}{\eta(t_2)}\right),
\end{equation}
  \item [(iv)]  $a\leqslant \Phi_2(a,0)$  for $a>0$.
\end{itemize}
Then $\Phi_2$ is a triangle function for the space $(Y,\rho)$.
\end{theorem}

\begin{proof}
Let $x',y',z' \in Y$ be different points and let $x=f^{-1}(x'), y=f^{-1}(y'), z=f^{-1}(z')$. Hence, $d(x,y)\leqslant \Phi_1(d(x,z),d(z,y))$ and by (ii)
\begin{equation}\label{e28}
1\leqslant \Phi_1\left(\frac{d(x,z)}{d(x,y)},\frac{d(z,y)}{d(x,y)}\right).
\end{equation}
Set
$$
\frac{d(x,y)}{d(x,z)}= t_1, \quad \frac{d(x,y)}{d(z,y)}= t_2.
$$
Hence,
\begin{equation}\label{e29}
1\leqslant \Phi_1 \left(\frac{1}{t_1}, \frac{1}{t_2}\right).
\end{equation}
By~(\ref{e0}) we have
$$
\rho(f(x),f(y))\leqslant \eta(t_1) \rho(f(x),f(z)), \quad
\rho(f(x),f(y))\leqslant \eta(t_2) \rho(f(z),f(y))
$$
or equivalently,
\begin{equation}\label{e210}
\rho(x',y')\leqslant \eta(t_1) \rho(x',z'), \quad
\rho(x',y')\leqslant \eta(t_2) \rho(z',y').
\end{equation}

By~(\ref{e29}),~(\ref{e27}), (i) and ~(\ref{e210}) we have
$$
1\leqslant \Phi_2 \left(\frac{1}{\eta(t_1)}, \frac{1}{\eta(t_2)}\right)
\leqslant \Phi_2 \left(\frac{\rho(x',z')}{\rho(x',y')}, \frac{\rho(z',y')}{\rho(x',y')}\right) \stackrel{\text{(ii)}}{\leqslant}\frac{1}{\rho(x',y')} \Phi_2(\rho(x',z'),\rho(z',y')).
$$
Hence, the inequality
\begin{equation}\label{i34}
\rho(x',y')\leqslant \Phi_2(\rho(x',z'),\rho(z',y'))
\end{equation}
 follows.

In the case when among the points $x', y', z' \in Y$ there are at least two equal points inequality~(\ref{i34}) easily follows from condition (iv).
\end{proof}

\begin{remark}\label{r23}
The most important triangle functions $\Phi(u,v)$ which generate well-known type of metrics and their generalizations are $u+v$ (metric), $K(u+v)$ ($b$-metric with $K\geqslant 1$), $\max\{u,v\}$ (ultrametric). Note that this proposition describes also cases when $\Phi_1$ and $\Phi_2$ are different triangle functions.
\end{remark}

If the usual triangle inequality is replaced by $d(x,y)\leqslant K(d(x,z)+d(z,y))$, $K\geqslant 1$, then  $(X,d)$ is called a \emph{b-metric space}. Initially, the definition of a b-metric space was introduced by Czerwik~\cite{C93} in 1993 as above only with the fixed $K=2$. After that, in 1998, Czerwik~\cite{C98} generalized this notion where the constant 2 was replaced by a constant $K\geqslant 1$, also with the same name b-metric.

\begin{corollary}\label{c3.1}
Let $(X,d)$ be a b-metric space with the coefficient $K_1$, $(Y,\rho)$ be a semimetric space, and let $f\colon X\to Y$ be a surjective $\eta$-quasisymmetric mapping. If  there exists $K_2\geqslant 1$ such that for all $t_1, t_2\in \RR^+\setminus\{0\}$
\begin{equation*}
\left( 1 \leqslant K_1\left(\frac{1}{t_1} + \frac{1}{t_2}\right) \right)
\Rightarrow
\left( 1 \leqslant K_2\left(\frac{1}{\eta(t_1)}+\frac{1}{\eta(t_2)}\right) \right),
\end{equation*}
then $\rho$ is a b-metric with the coefficient $K_2$.
\end{corollary}
\begin{proof}
It suffices to set in Theorem~\ref{p24} $\Phi_1(x,y)=K_1(x+y)$, $\Phi_2(x,y)=K_2(x+y)$.
\end{proof}

\begin{corollary}
In the case $K_1=K_2=1$ we the have a sufficient condition which guarantees that $f$ is metric preserving.
\end{corollary}

Recall that an \emph{ultrametric} is a metric for which the strong triangle inequality $d(x, y)\leqslant \max \{d(x, z), d(z, y)\}$ holds. In this case the pair $(X,d)$ is called an \emph{ultrametric space}. Note that, the ultrametric inequality was formulated by F.~Hausdorff in 1934 and ultrametric spaces were introduced by M. Krasner~\cite{Kr44} in 1944.

\begin{corollary}\label{c3.5}
Let $(X,d)$ be an ultrametric space, $(Y,\rho)$ be a semimetric space and let $f\colon X\to Y$ be a surjective $\eta$-quasisymmetric mapping such that $\eta(1)=1$. Then $(Y,\rho)$ is also an ultrametric space.
\end{corollary}
\begin{proof}
Consider Theorem~\ref{p24} with the triangle functions $\Phi_1(u,v)=\Phi_2(u,v)=\max\{u,v\}$. Clearly, conditions (i) and (ii) hold. Let us show condition (iii), i.e., that
 for all $t_1, t_2\in \RR^+$ the inequality
\begin{equation*}
\max \left\{ \frac{1}{t_1},\frac{1}{t_2}\right\}\geqslant 1
\, \text{ implies } \,
\max\left\{\frac{1}{\eta(t_1)},\frac{1}{\eta(t_2)}\right\}\geqslant 1.
\end{equation*}
In other words it is enough to show that if one of the numbers $t_1, t_2$ is smaller or equal to 1, then one of the numbers $\eta(t_1), \eta(t_2)$ is also smaller or equal to 1. But this easily follows from the fact that $\eta$ is strictly increasing and $\eta(0)=0$, $\eta(1)=1$.
\end{proof}

Corollaries~\ref{c3.1} and~\ref{c3.5} were initially obtained in~\cite{PS21B} and~\cite{PS21U}, respectively, by direct proofs. See also~\cite{PS21B} and~\cite{PS21U} for more results related to  quasisymmetric embeddings between b-metric spaces and ultrametric spaces, respectively.

\subsection{Ptolemy's inequality preserving mappings}
A metric space $(X,d)$ is called Ptolemaic if for all $x, y, z, t \in X$ the inequality
\begin{equation}\label{e270}
d(x,z)d(t,y)\leqslant d(x,y)d(t,z)+d(x,t)d(y,z)
\end{equation}
holds, see, e.g.,~\cite{Sch40, Sch52}. Every pre-Hilbert space is Ptolemaic (see~\cite[9.7.3.8, 10.9.2]{Ber09} for instance), and each linear quasinormed Ptolemaic space is a pre-Hilbert space~\cite{Sch52}.  The Ptolemy theorem, known since antiquity, states that~(\ref{e270}) turns into equality when $x$, $y$, $z$, and $t$ are the vertices of a convex quadrilateral inscribed into a circle. Ptolemaic spaces still attract attention of many mathematicians, see e.g.,~\cite{As18, DP11, MS13, BFW09}.

In what follows under Ptolemaic spaces we understand semimetric spaces $(X,d)$ for which inequality~(\ref{e270}) holds. Note that~(\ref{e270}) does not imply the standard triangle inequality in $(X,d)$.

\begin{proposition}\label{p3.11}
Let $(X,d)$ be a Ptolemaic space, $(Y,\rho)$ be a semimetric space and let $f\colon X\to Y$ be a surjective $\eta$-quasisymmetric mapping. If for all $t_1, t_2, t_3, t_4\in \RR^+$ the inequality
\begin{equation}\label{e271}
t_1t_2t_3t_4 \leqslant t_1t_2+t_3t_4
\, \text{ implies } \,
\eta(t_1)\eta(t_2)\eta(t_3)\eta(t_4)\leqslant \eta(t_1)\eta(t_2)+\eta(t_3)\eta(t_4),
\end{equation}
then $(Y,\rho)$ is also Ptolemaic.
\end{proposition}

\begin{proof}
Let $x',y',z',t' \in Y$ be different points and let $x=f^{-1}(x'), y=f^{-1}(y'), z=f^{-1}(z'), t=f^{-1}(t')$.  By~(\ref{e270}) we have
\begin{equation}\label{e281}
1\leqslant \frac{d(x,y)}{d(x,z)}\frac{d(t,z)}{d(t,y)}+\frac{d(x,t)}{d(x,z)}\frac{d(y,z)}{d(t,y)}.
\end{equation}
Set
$$
\frac{d(x,z)}{d(x,y)}= t_1, \quad \frac{d(t,y)}{d(t,z)}= t_2, \quad \frac{d(x,z)}{d(x,t)}= t_3, \quad \frac{d(t,y)}{d(y,z)}= t_4.
$$
Hence,
\begin{equation}\label{e291}
1\leqslant \frac{1}{t_1}\frac{1}{t_2}+\frac{1}{t_3}\frac{1}{t_4}.
\end{equation}

By~(\ref{e0}) we have
$$
\rho(f(x),f(z))\leqslant \eta(t_1) \rho(f(x),f(y)), \quad
\rho(f(t),f(y))\leqslant \eta(t_2) \rho(f(t),f(z)),
$$
$$
\rho(f(x),f(z))\leqslant \eta(t_3) \rho(f(x),f(t)), \quad
\rho(f(t),f(y))\leqslant \eta(t_4) \rho(f(y),f(z)).
$$
Hence,
$$
\frac{1}{\eta(t_1)}\frac{1}{\eta(t_2)}+
\frac{1}{\eta(t_3)}\frac{1}{\eta(t_4)}
\leqslant \frac{\rho(x',y')}{\rho(x',z')}\frac{\rho(t',z')}{\rho(t',y')}+
\frac{\rho(x',t')}{\rho(x',z')}\frac{\rho(y',z')}{\rho(t',y')}.
$$
Using~(\ref{e291}) and condition~(\ref{e271}) we obtain inequality~(\ref{e270}) for  $x', y', z', t'$,  which completes the proof.

In the case when among the points $x',y',z',t' \in Y$ there are at least two different points inequality~(\ref{e270}) is almost evident.
\end{proof}

The following assertion is well-known, see, e.g., Section 2.12 in~\cite{HLP52}.
\begin{lemma}\label{l3.1}
If $0<\alpha\leqslant 1$, then for $u,v\geqslant 0$ the inequality
$(u+v)^{\alpha}\leqslant u^{\alpha}+v^{\alpha}$
holds.
\end{lemma}

\begin{corollary}\label{c214}
Let $X$ be a Ptolemaic space, $Y$ be a semimetric space and let $f\colon X\to Y$ be a surjective $\eta$-quasisymmetric mapping such that $\eta(t)=t^{\alpha}$, $0<\alpha\leqslant 1$. Then $Y$ is also a Ptolemaic space.
\end{corollary}

\begin{proof}
Let $t_1, t_2, t_3, t_4$ belong to $\RR^+$ and let the inequality
\begin{equation}\label{e6}
t_1t_2t_3t_4\leqslant t_1t_2+t_3t_4
\end{equation}
hold. Then~(\ref{e6}) implies $t_1^{\alpha}t_2^{\alpha}t_3^{\alpha}t_4^{\alpha}\leqslant (t_1t_2+t_3t_4)^{\alpha}$ and, by Lemma~\ref{l3.1}, we have
$$
(t_1t_2+t_3t_4)^{\alpha}\leqslant t_1^{\alpha}t_2^{\alpha}+t_3^{\alpha}t_4^{\alpha}.
$$
Hence, $Y$ is Ptolemaic by Proposition~\ref{p3.11}.
\end{proof}

\subsection{Metric betweenness preserving mappings}

Let $(X,d)$ be a semimetric space and let $x,y,z$ be different points from $X$. We shall say that the point $y$ lies between $x$ and $z$ if the equality
\begin{equation}\label{l0}
d(x,z)=d(x,y)+d(y,z)
\end{equation}
holds.  This relation is intuitive for points belonging to some straight line, plane or three-dimensional space. K. Menger~\cite[p.~77]{Me28} seems to be the first who formulated the concept of ``metric betweenness'' for general metric
spaces.

Let $X$ and $Y$ be semimetric spaces and let $f\colon X\to Y$ be a mapping.
If $f(y)$ lies between $f(x)$ and $f(z)$ in $Y$ whenever $y$ lies between $x$ and $z$ in $X$, then we say that $f$ preserves metric betweenness.
Note also that equality~(\ref{l0}) implies that the subset $\{x,y,z\}\subseteq X$ is isometrically embeddable in $\RR^1$.

\begin{lemma}\label{l02}
Let $(X,d)$, $(Y,\rho)$ be semimetric spaces and let $f\colon X\to Y$ be an $\eta$-quasisymmetric mapping. If the equalities
\begin{equation}\label{l1}
\frac{1}{\eta\left(\frac{1}{t_1}\right)}+\frac{1}{\eta\left(\frac{1}{t_2}\right)} = 1 \, \,  \text{ and } \, \,
\eta(t_1)+\eta(t_2)= 1
\end{equation}
hold whenever $t_1, t_2\in (0,\infty)$ and $t_1+t_2=1$, then $f$ preserves metric betweenness.

Conversely, if $f$ preserves metric betweenness, then \begin{equation}\label{l2}
\frac{1}{\eta\left(\frac{1}{t_1}\right)}+\frac{1}{\eta\left(\frac{1}{t_2}\right)} \leqslant 1 \, \, \text{ and } \, \,
\eta(t_1)+\eta(t_2)\geqslant 1
\end{equation}
hold whenever there are distinct $x, y, z \in X$ satisfying~(\ref{l0}) such that $d(x,y)=t_1d(x,z)$ and $d(y,z)=t_2d(x,z)$.
\end{lemma}

\begin{proof}
Suppose~(\ref{l1}) is valid whenever $t_1, t_2\in (0,\infty)$ and $t_1+t_2=1$ and let equality~(\ref{l0}) hold.
We have to show that
\begin{equation}\label{emb}
\rho(f(x), f(z))=\rho(f(x), f(y))+\rho(f(y), f(z)).
\end{equation}

Let $t_1, t_2 >0$ be such that $d(x,y)=t_1d(x,z)$, $d(y,z)=t_2d(x,z)$. Then $t_1+t_2=1$. It follows from~(\ref{e0}) that
\begin{equation}\label{k4}
\begin{split}
&\rho(f(x), f(y))\leqslant \eta(t_1)\rho(f(x), f(z)), \\
&\rho(f(y), f(z))\leqslant \eta(t_2)\rho(f(x), f(z))
\end{split}
\end{equation}
and
\begin{equation}\label{k5}
\begin{split}
&\rho(f(x), f(z))\leqslant \eta\left(\frac{1}{t_1}\right)\rho(f(x), f(y)) \\
&\rho(f(x), f(z))\leqslant \eta\left(\frac{1}{t_2}\right)\rho(f(y), f(z)).
\end{split}
\end{equation}
Hence,~(\ref{k4}) implies
\begin{equation}\label{k6}
\rho(f(x), f(y))+\rho(f(y), f(z))\leqslant (\eta(t_1)+\eta(t_2))\rho(f(x), f(z))
\end{equation}
and~(\ref{k5}) implies
\begin{equation}\label{k7}
\rho(f(x), f(y))+\rho(f(y), f(z))\geqslant \rho(f(x), f(z))\left(\frac{1}{\eta\left(\frac{1}{t_1}\right)}
+\frac{1}{\eta\left(\frac{1}{t_2}\right)}\right).
\end{equation}
Clearly,~(\ref{l1}), ~(\ref{k6}) and ~(\ref{k7}) imply equality~(\ref{emb}).

Conversely, let $f$ preserve metric betweenness. Then for all $x, y, z$ satisfying~(\ref{l0}) equality~(\ref{emb}) holds.  Let $t_1, t_2$ be such that  $d(x,y)=t_1d(x,z)$ and $d(y,z)=t_2d(x,z)$.
It is clear that inequalities~(\ref{k4}), (\ref{k5}), (\ref{k6}), and~(\ref{k7}) hold again. Taking into consideration equality~(\ref{emb}) we see that~(\ref{l2}) follows from~(\ref{k6}) and~(\ref{k7}).
\end{proof}

Recall that a function $\Phi(x,y)$ of two variables is antisymmetric if $\Phi(y,x)=-\Phi(x,y)$.

\begin{theorem}\label{t226}
Let $(X,d)$, $(Y,\rho)$ be semimetric spaces and let $f\colon X\to Y$ be an $\eta$-quasisymmetric mapping. If the homeomorphism $\eta$ has the form
\begin{equation}\label{k8}
\eta(t)=\begin{cases}
\frac12+\Psi_1(t,1-t),& t\in [0,1],\\
\frac{1}{\frac12+\Psi_2\left(\frac1t, 1-\frac1t\right)},&t\in [1, \infty),
\end{cases}
\end{equation}
where $\Psi_1$, $\Psi_2$ are some continuous, antisymmetric, strictly increasing in the first variables, defined on $[0,1]\times[0,1]$ functions of two variables such that $\Psi_1(1,0)=\Psi_2(1,0)=1/2$,  then $f$ preserves metric betweenness.
\end{theorem}

\begin{proof}
According to the first part of Lemma~\ref{l02} it suffices to establish equalities~(\ref{l1}) for $\eta$ given by~(\ref{k8}) whenever $t_1, t_2 \in (0,\infty)$ and $t_1+t_2=1$ and to show that $\eta$ is, indeed, a homeomorphism with $\eta(1)\geqslant 1$. In other words the following two systems must hold:
\begin{equation}\label{l11}
\begin{cases}
\eta(t_1)+\eta(t_2)= 1,\\
t_1+t_2=1,\\
t_1,t_2 \in (0,\infty),
\end{cases}
\end{equation}
and
\begin{equation}\label{l12}
\begin{cases} \frac{1}{\eta\left(\frac{1}{t_1}\right)}+\frac{1}{\eta\left(\frac{1}{t_2}\right)} = 1,\\
t_1+t_2=1,\\
t_1,t_2 \in (0,\infty).
\end{cases}
\end{equation}

Let us consider system~(\ref{l11}) on the extended set $t_1,t_2 \in [0,\infty)$. Set $t_1=t$. Then~(\ref{l11}) is reduced to the functional equation $\eta(t)+\eta(1-t)=1$, $t\in [0,1]$, which has a general solution
\begin{equation}\label{l13}
\eta(t)=\frac{1}{2}+\Psi_1(t,1-t),
\end{equation}
see~\cite{PM98}, where $\Psi_1(u,v)=-\Psi_1(v,u)$ is any antisymmetric function with two arguments. Clearly, $\eta$ given by~(\ref{l13}) also satisfies~(\ref{l11}) in the case $t_1,t_2 \in (0,\infty)$.

In~(\ref{l12}) set $t_1=s$, $s\in (0,1)$. Then $\frac{1}{t_2}=\frac{1}{1-s}$ and
$$
\frac{1}{\eta\left(\frac{1}{s}\right)}+\frac{1}{\eta\left(\frac{1}{1-s}\right)}=1.
$$
Substitution $\frac{1}{\eta\left(\frac{1}{s}\right)}=F(s)$ gives the equality
$F(s)+F(1-s)=1$, which as above has the general solution
$$
F(s)=\frac{1}{2}+\Psi_2(s,1-s),
$$
where $\Psi_2(u,v)=-\Psi_2(v,u)$ is any antisymmetric function of two variables.
Hence,
$$
\eta\left(\frac{1}{s}\right)=\frac{1}{\frac{1}{2}+\Psi_2(s,1-s)}.
$$
Substitution $t=\frac{1}{s}$, $t\in (1,\infty)$ gives the solution
\begin{equation}\label{l14}
\eta(t)=\frac{1}{\frac{1}{2}+\Psi_2\left(\frac1t,1-\frac{1}{t}\right)}.
\end{equation}

Since $\eta$ is a homeomorphism consider that $\Psi_1$ and $\Psi_2$ are continuous.

From~(\ref{l13}) the condition $\eta(0)=0$ gives $\Psi_1(0,1)=-\frac{1}{2}$ and by antisymmetry $\Psi_1(1,0)=\frac{1}{2}$. Again, from~(\ref{l13}) we have $\eta(1)=\frac{1}{2}+\Psi_1(1,0)=1$. Using this equality and~(\ref{l14}) we impose the condition $\Psi_2(1,0)=\frac{1}{2}$ on the function $\Psi$ and by the symmetry also $\Psi_2(0,1)=-\frac{1}{2}$. Thereby we extent the solution~(\ref{l14}) to the interval $[1,\infty)$.

Let us show that if the functions $\Psi_1, \Psi_2$ are strictly increasing in the first variable then $\eta(t)$ is also strictly increasing. Let $t_1<t_2$, $t_1, t_2\in [0,1]$. Then $-\Psi_1(v, t_1)=\Psi_1(t_1,v)<\Psi_1(t_2,v)=-\Psi_1(v,t_2)$, i.e., $\Psi_1(v,t_2)<\Psi_1(v,t_1)$. Thus, $\Psi_1$ is strictly decreasing in the second variable. Analogously, we have the same for $\Psi_2$.
Since $\Psi_2$ is continuous in two variables, we have
$$
\lim\limits_{t\to \infty}\Psi_2\left(\frac{1}{t}, 1-\frac{1}{t}\right)=-\frac12.
$$
Hence,~(\ref{l13}) and~(\ref{l14}) imply that $\eta(t)$ is strictly increasing on $[0,\infty)$.

Thus, $\eta(t)$ given by~(\ref{k8}), $t\in [0,\infty)$,  satisfies equalities~(\ref{l1}), whenever $t_1+t_2=1$, $t_1,t_2 \in (0,\infty)$ and hence by Lemma~\ref{l02} $f$ preserves metric betweenness.
\end{proof}

\begin{example}
Let $f_1, f_2\colon [0,1]\to \mathbb R^+$ be strictly increasing, continuous, differentiable on $(0,1)$ functions such that $f_{1,2}(1)=1/2$, $f_{1,2}(0)=0$. Consider the functions
$$
\Psi_1(x,y)=f_1(x)-f_1(y),
$$
$$
\Psi_2(x,y)=f_2(x)-f_2(y).
$$
It is clear that $\Psi_1$ and $\Psi_2$ satisfy conditions of Theorem~\ref{t226}. Let us show that $\eta$ is a homeomorphism. It is easy to see that $\eta(0)=0$ and that $\eta$ is continuous. It suffices to show that $\eta$ is strictly increasing. For $t\in (0,1)$ we have
$$
\eta'(t)=f_1'(t)+f_1'(1-t)>0.
$$
If $t\in (1,+\infty)$, then
$$
\eta'(t)=\frac{-((-1/t^2)f_2'(1/t)-(1/t^2)f_2'(1-1/t))}
{(1/2+f_2(1/t)-f_2(1-1/t))^2}
$$
$$
=\frac{(1/t^2)f_2'(1/t)+(1/t^2)f_2'(1-1/t)}
{(1/2+f_2(1/t)-f_2(1-1/t))^2}>0.
$$
Thus, $\eta$ is strictly increasing.

For example, as $f_{1}, f_{2}$ we can take functions $x^{n}/2$, $n \in \mathbb N$, in any combinations.
\end{example}

Recall that a four-point metric space $(X,d)$ is called a \emph{pseudolinear quadruple} if there exists an enumeration $x_1, x_2, x_3, x_4$ of the points of $X$ such that the equalities
\begin{gather}\label{e21}
d(x_1,x_2)=d(x_3,x_4)=t, \, \, d(x_2,x_3)=d(x_4,x_1)=s, \\
d(x_2,x_4)=d(x_3,x_1)=s+t \notag
\end{gather}
hold with some positive reals $s$ and $t$, see Figure~\ref{fig1}. Note also that equilateral pseudolinear quadruples are known by their extremal properties~\cite{DP11}.

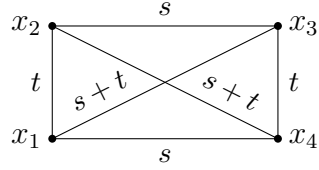
\begin{figure}[ht]
\begin{center}
\begin{tikzpicture}[scale=0.5]
\draw (0,1.5) node [left] {\small{$t$}};
\draw (6,1.5) node [right] {\small{$t$}};
\draw (3,0) node [below] {\small{$s$}};
\draw (3,3) node [above] {\small{$s$}};

\draw (0,0) node [below,left] {$x_1$} --
      (0,3) node [left] {$x_2$} --
      (6,3) node [right] {$x_3$} --
      (6,0) node [right] {$x_4$} --
      (0,0);
 \foreach \i in {(0,0),(0,3),(6,3),(6,0)}
  \fill[black] \i circle (3pt);
\draw (0,0) -- (6,3) node [sloped, near start, above] {\small{$s+t$}};
\draw (0,3) -- (6,0) node [sloped, near end, above] {\small{$s+t$}};
\end{tikzpicture}
\caption{The metric space $(X,d)$.}
\label{fig1}
\end{center}
\end{figure}

In 1928 K. Menger proved the following.
\begin{theorem}[\cite{Me28}]\label{t1}
If every three points of a metric space $X$, $\card (X)\geqslant 3$, are isometrically embeddable in $\RR^1$, then $X$ is isometric to some subset of $\RR^1$ or $X$ is a pseudolinear quadruple.
\end{theorem}

Note that the requirement for $X$ to be a metric space is redundant. The theorem could be formulated for semimetric spaces since the embeddability of every three points into the real line automatically implies that $X$ is a metric space.

The following assertion is formulated for arbitrary mapping $f$. But it holds in particular for an $\eta$-{quasisymmetric} $f$ satisfying conditions of Theorem~\ref{t226} or the first condition of Lemma~\ref{l02}.

\begin{proposition}
Let $(X,d)$, $(Y,\rho)$ be semimetric spaces and let $f\colon X\to Y$ be a mapping preserving metric betweenness. If $A\subset X$  is isometrically embeddable in $\RR^1$, then $f(A)$ is also isometrically embeddable in $\RR^1$. Moreover, if $A\subset X$ is a pseudolinear quadruple, then $f(A)$ is also a pseudolinear quadruple.
\end{proposition}

\begin{proof}
For $\card (A) \neq 4$ the first assertion follows directly from Theorem~\ref{t1}. If $\card (A)=4$ and $A$ is isometrically embeddable in $\RR^1$, then $A$ has only one diametrical pair of points $\{a,b\}$. Clearly, $\{f(a),f(b)\}$ is a single diametrical pair in $f(A)$. Since every pseudolinear quadruple has two diametrical pairs of points, by Theorem~\ref{t1} we have that $f(A)$ is isometric to some subset of $\RR^1$

Let $A=\{x_1, x_2, x_3, x_4\}$ and let condition~(\ref{e21}) hold. Since $f$ preserves metric betweenness, we have
\begin{gather*}
\rho(f(x_1),f(x_2))=a, \, \, \rho(f(x_2),f(x_3))=b, \, \, \rho(f(x_3),f(x_4))=c, \, \, \rho(f(x_4),f(x_1))=d, \\
\rho(f(x_2),f(x_4))=b+c=a+d, \, \, \rho(f(x_1),f(x_3))=a+b=c+d
\end{gather*}
with some positive reals $a,b,c,d$.

The system of equalities
\begin{equation*}
\begin{cases}
b+c=a+d,\\
a+b=c+d,
\end{cases}
\end{equation*}
implies that $d=b$ and $c=a$. Hence, $\rho(f(x_2),f(x_4))= \rho(f(x_1),f(x_3))=a+b$ and $\{f(x_1), f(x_2), f(x_3), f(x_4)\}$ is a pseudolinear quadruple in $(Y,\rho)$.
\end{proof}

\section{A generalization of the Tukia-V\"{a}is\"{a}l\"{a} inequality}\label{s4}

The following proposition was proved in~\cite{TV80} by P. Tukia and J. V\"{a}is\"{a}l\"{a}, see also Propositon 10.8 in~\cite{H01} for the extended proof.

\begin{proposition}\label{p1.2}
Let $X, Y$ be metric spaces and let $f$ be an $\eta$-quasi\-sym\-metric mapping. Let $A\subset B\subset X$ with $\diam A>0$, $\diam B<\infty$. Then $\diam f(B)<\infty$ and
\begin{equation}\label{e11}
  \frac{1}{2\eta\left(\frac{\diam B}{\diam A}\right)}\leqslant
  \frac{\diam f(A)}{\diam f(B)}\leqslant
  \eta\left( \frac{2\diam A}{\diam B}\right).
\end{equation}
\end{proposition}
In Theorem~\ref{p2}  we prove a generalization of this proposition  in the case when $X$ and $Y$ are semimetric spaces with different triangle functions.

The following assertion was formulated in~\cite{TV80} for the case when $X$ and $Y$ are metric spaces but we prove it for the case when $X$ and $Y$ are semimetric spaces.
\begin{proposition}\label{p1}
Let $(X,d)$ and $(Y,\rho)$ be semimetric spaces.
If $f\colon X\to Y$ is an $\eta$-quasisymmetric mapping, then $f^{-1}\colon f(X)\to X$ is an $\eta'$-quasisymmetric mapping, where
\begin{equation}\label{e81}
\eta'(t) = 1 / \eta^{-1}(t^{-1})
\end{equation}
 for $t>0$.
\end{proposition}

\begin{proof}
Let $a_1, b_1,x_1 \in f(X)$ and let $a=f^{-1}(a_1)$, $b=f^{-1}(b_1)$, and $x=f^{-1}(x_1)$. Let us prove the proposition by contradiction. Assume that
$$
\rho(x_1,a_1)\leqslant t\rho(x_1,b_1) \, \text{ but } \,  d(x,a)> \eta'(t)d(x,b).
$$
Then by~(\ref{e81}) we have
$d(x,b)< \eta^{-1}(\frac{1}{t})\rho(x,a)$. Using~(\ref{e0}) we get $\rho(x_1,b_1)< \frac{1}{t}\rho(x_1,a_1)$, which contradicts our assumption.
\end{proof}

\begin{theorem}\label{p2}
Let $(X,d)$ and $(Y,\rho)$ be semimetric spaces with continuous and strictly increasing in both of their arguments triangle functions $\Phi_1$ and $\Phi_2$, respectively. And let $f\colon X\to Y$ be an $\eta$-quasisymmetric mapping. Then $f$ maps bounded subspaces to bounded subspaces.

Moreover, if $A\subseteq B\subseteq X$, $0< \diam A,  \diam B <\infty$,  then $\diam f(B)$ is finite and
\begin{equation}\label{e12}
  \frac{\diam f(A)}{\diam f(B)}\leqslant
  \eta\left( \frac{\diam A}{\varphi_1^{-1}(\diam B)}\right),
\end{equation}

\begin{equation}\label{e122}
\frac{1}{ \eta\left( \frac{\diam B}{\diam A}\right)}
\leqslant
\frac{\diam f(A)}{\varphi_2^{-1}(\diam f(B))},
\end{equation}
where $\varphi_1(t)=\Phi_1(t,t)$, $\varphi_2(t)=\Phi_2(t,t)$.
\end{theorem}

\begin{proof}
Let $(b_n)$ and $(b'_n)$ be sequences such that ${1}/{2}\diam B \leqslant d(b_n,b'_n)$ for all $n$ and
$$
d(b_n,b'_n)\to \diam B, \text{ as } n \to \infty.
$$
For every $x\in B$ we have
$$
d(x,b_1)\leqslant \diam B \leqslant 2d(b_1,b'_1)
$$
by~(\ref{e0}) implying
$$
\rho(f(x),f(b_1))\leqslant \eta(2)\rho(f(b_1),f(b'_1)).
$$
In order to see that $\diam f(B)<\infty$ for any $x,y \in B$ consider the inequalities
$$
\rho(f(x),f(y)) \leqslant \Phi_2(\rho(f(x),f(b_1)), \rho(f(y),f(b_1)))
$$
$$
\leqslant  \Phi_2(\eta(2)\rho(f(b_1),f(b'_1)), \eta(2)\rho(f(b_1),f(b'_1))).
$$

Let $x, a\in A$. To prove inequality~(\ref{e12})
consider the evident inequality
$$
d(a,x)\leqslant \frac{d(x,a)}{d(b_n,a)}d(a,b_n),
$$
which by~(\ref{e0}) implies
\begin{equation}\label{e13}
\rho(f(x),f(a))\leqslant \eta \left( \frac{d(x,a)}{d(b_n,a)} \right)\rho(f(b_n),f(a)).
\end{equation}

Without loss of generality (if needed swapping $b_n$ and $b'_n$) we may assume that
$$
d(b'_n,a)\leqslant d(b_n,a).
$$
Using the triangle inequality
$$
d(b_n, b'_n)\leqslant \Phi_1(d(b_n, a), d(a, b'_n))
$$
and the monotonicity of $\Phi_1$, we get
$$
d(b_n, b'_n)\leqslant \varphi_1(d(b_n, a)),
$$
where $\varphi_1(t)=\Phi_1(t,t)$.
Hence,
$$
\varphi_1^{-1}(d(b_n, b'_n))\leqslant d(b_n, a).
$$

Using the last inequality and the relations $d(x,a)\leqslant \diam A$, $A\subseteq B$, from~(\ref{e13}) we have
$$
\rho(f(x),f(a))\leqslant \eta \left( \frac{\diam A}{\varphi_1^{-1}(d(b_n, b'_n))} \right)\diam f(B).
$$

Since $ d(b_n,b'_n)\to \diam B$ we have inequality~(\ref{e12}).

By Proposition~\ref{p1} $f^{-1}\colon f(X) \to X$ is an $\eta'$-quasisymmetric mapping.
Since $f(A)\subseteq f(B)\subseteq f(X)$, $0<\diam f(A), \diam f(B)< \infty$, applying inequality~(\ref{e12}) to $f^{-1}$  we have
$$
\frac{\diam A}{\diam B}\leqslant \eta'\left( \frac{\diam f(A)}{\varphi_2^{-1}(\diam f(B))}\right).
$$
From~(\ref{e81}) we have
$$
\frac{\diam A}{\diam B}\leqslant \left(\eta^{-1}\left( \frac{\varphi_2^{-1}(\diam f(B))}{\diam f(A)}\right)\right)^{-1}.
$$
Hence,
$$
\eta^{-1}\left( \frac{\varphi_2^{-1}(\diam f(B))}{\diam f(A)}\right) \leqslant \frac{\diam B}{\diam A}
$$
and, since $\eta$ is strictly increasing, we have
$$
\frac{\varphi_2^{-1}(\diam f(B))}{\diam f(A)} \leqslant \eta\left( \frac{\diam B}{\diam A}\right),
$$
which implies inequality~(\ref{e122}). This completes the proof.
\end{proof}

\begin{corollary}\label{c4.3}
Let $X$ and $Y$ be $b$-metric spaces with the coefficients $K_1$ and $K_2$, respectively. Then the double inequality holds
\begin{equation}\label{e4}
  \frac{1}{2K_2\eta\left(\frac{\diam B}{\diam A}\right)}\leqslant
  \frac{\diam f(A)}{\diam f(B)}\leqslant
  \eta\left( \frac{2K_1\diam A}{\diam B}\right).
\end{equation}
\end{corollary}

\begin{proof}
By Remark~\ref{r23} we have  $\Phi_1(x,y)=K_1(x+y)$ and $\Phi_2(x,y)=K_2(x+y)$. Using Theorem~\ref{p2} with $\varphi_1(x)=2K_1x$, $\varphi_1^{-1}(x)=x/(2K_1)$ and  $\varphi_2(x)=2K_2x$, $\varphi_2^{-1}(x)=x/(2K_2)$ we obtain double inequality~(\ref{e4}).
\end{proof}

\begin{corollary}
Let $X$ and $Y$ be $b$-metric spaces with the coefficients $K_1$ and $K_2$, respectively and let $f\colon X\to Y$ be an $\eta$-quasisymmetric mapping. Then
\begin{equation*}
\eta(2K_1r)2K_2\eta \left( \frac{1}{r} \right)\geqslant 1,
\end{equation*}
where $r=\diam A / \diam B$, $A\subseteq B\subseteq X$, $0< \diam A,  \diam B <\infty$.
\end{corollary}

Finally, we get Proposition~\ref{p1.2} as a following corollary.
\begin{corollary}
If $X$ and $Y$ are metric spaces, then double inequality~(\ref{e4}) holds with $K_1=K_2=1$.
\end{corollary}

\begin{proof}
It suffices to apply Corollary~\ref{c4.3} with $K_1=K_2=1$, since the $b$-metric spaces $X$ and $Y$ with the coefficients $K_1$ and $K_2$ are metric spaces if and only if $K_1=K_2=1$.
\end{proof}

\begin{corollary}\label{c4.6}
If $X$ and $Y$ are ultrametric spaces, then double inequality~(\ref{e4}) holds with $K_1=K_2=\frac{1}{2}$.
\end{corollary}

\begin{proof}
By Remark~\ref{r23} we have $\Phi_1(x,y)=\Phi_2(x,y)=\max\{x,y\}$. Hence, it suffices to apply Theorem~\ref{p2} with $\varphi_1(x)=\varphi_{2}(x)=x$.
\end{proof}

\begin{remark}
Corollaries~\ref{c4.3} and~\ref{c4.6} were initially obtained by direct proofs in~\cite{PS21B} and \cite{PS21U}, respectively.
\end{remark}

\begin{corollary}\label{cc2}
Let $(X,d)$ and $(Y,\rho)$ be semimetric spaces with continuous and strictly increasing in both of their arguments triangle functions $\Phi_1$ and $\Phi_2$, respectively. Let $f\colon X\to Y$ be an $\eta$-quasisymmetric mapping and let $(X,d)$ be a bounded space. Then $(f(X),\rho)$ is bounded and
\begin{equation}\label{e128}
\frac{\varphi_2^{-1}(\diam f(X))}{ \eta\left( \frac{\diam X}{d(x,y)}\right)}
\leqslant
\rho(f(x),f(y))\leqslant
\diam Y   \eta\left( \frac{d(x,y)}{\varphi_1^{-1}(\diam X)}\right),
\end{equation}
where $x,y \in X$, $x\neq y$,  and $\varphi_1(t)=\Phi_1(t,t)$, $\varphi_2(t)=\Phi_2(t,t)$.
\end{corollary}

\begin{proof}
It suffices to set $B=X$ and $A=\{x,y\}$  in Theorem~\ref{p2}.
\end{proof}

Corollary~\ref{cc2} implies the following.
\begin{corollary}\label{cc3}
Let $(X,d)$ and $(Y,\rho)$ be metric spaces, $f\colon X\to Y$ be a $Ct$-quasi\-sym\-metric mapping, $C>0$, such that $f(X)=Y$ and let $(X,d)$ be a bounded space. Then $(Y,\rho)$ is also bounded and
$f$ is $L$-bi-Lipshitz mapping with
$$L=2C\max\{\diam Y / \diam X, \diam X / \diam Y\}.$$
\end{corollary}

The following proposition generalizes Theorem~2.24 from~\cite{TV80}. Definitions of Cauchy sequence, completeness and totally boundedness are the same as in metric spaces.
\begin{proposition}\label{p26}
Let $(X,d)$ and $(Y,\rho)$ be semimetric spaces with continuous and strictly increasing in both of their arguments triangle functions. Then every quasisymmetric mapping $f\colon X\to Y$ maps Cauchy sequences to Cauchy sequences. If $X$ is totally bounded or complete, then $f(X)$ is totally bounded or complete, respectively.
\end{proposition}

\begin{proof}
The proof repeats the proof from~\cite{TV80} with the difference that we have to use the inequality
\begin{equation*}
\rho(f(x_i),f(x_j))\leqslant \eta \left( \frac{d(x_i,x_j)}{\diam(B)} \right)\diam (fB).
\end{equation*}
which follows directly from~(\ref{e13}).
\end{proof}

\section{Connections with weak similarities}\label{s5}
Recall that the \emph{spectrum} of a semimetric space $(X, d)$ is the set $$\Sp{X}=\{d(x,y)\colon x,y \in X\}.$$

\begin{definition}\label{d4.5}
Let  $(X,d)$ and $(Y,\rho)$ be semimetric spaces. A bijective mapping $f\colon X\to Y$ is a \emph{weak similarity} if there exists a strictly increasing bijection $\varphi\colon \Sp{X}\to \Sp{Y}$ such that the equality
\begin{equation}\label{e4.5}
\varphi(d(x,y))=\rho(f(x),f(y))
\end{equation}
holds for all $x$, $y\in X$. The function $\varphi$ is said to be a \emph{scaling function} of $f$. If $f\colon X\to Y$ is a weak similarity, we write $X \we Y$ and say that $X$ and  $Y$  are \emph{weakly similar}. The pair $(\varphi,f)$ is called a \emph{realization} of $X\we Y$.
\end{definition}

In~\cite{DP13} the notion of weak similarity was introduced in a slightly different but equivalent form, where also some properties of these mappings were studied. Weak similarities between finite ultrametric spaces were considered in~\cite{P18}.

\begin{remark}\label{r2}
The pair $(\varphi,f)$ is a realization of $(X,d)\we (Y,\rho)$ if and only if $(X,\varphi\circ d)$ and $(Y,\rho)$ are isometric with the isometry $f$.
\end{remark}

\begin{proposition}\label{tt2}
Let $(X,d)$ and $(Y,\rho)$ be semimetric spaces, $f\colon X\to Y$ be a bijection and let $X \we Y$  with the realization $(\varphi,f)$. If there exists a continuous submultiplicative strictly increasing continuation $\varphi^*\colon [0,\infty) \to [0,\infty)$ of the function  $\varphi$, then $f$ is an $\eta$-quasisymmetric mapping with $\eta(t)=\varphi^*(t)$.
\end{proposition}

\begin{proof}
Let $a,b,x \in X$ and let $t>0$ be such that the first inequality in~(\ref{e0}) holds. It follows from~(\ref{e4.5}) that
$$
d(x,a)=\varphi^{-1}(\rho(f(x),f(a)))), \quad
d(x,b)= \varphi^{-1}(\rho(f(x),f(b)))).
$$
Using these equalities and the first inequality in~(\ref{e0}) we obtain
$$
\varphi^{-1}(\rho(f(x),f(a))))\leqslant t \varphi^{-1}(\rho(f(x),f(b)))).
$$
Applying $\varphi^*$ to both parts we have
$$
\rho(f(x),f(a)))\leqslant \varphi^*(t \varphi^{-1}(\rho(f(x),f(b)))).
$$
Using submultiplicativity of $\varphi^*$ we obtain
$$
\rho(f(x),f(a)))\leqslant \varphi^*(t) \rho(f(x),f(b))).
$$
Comparing this inequality with the second inequality in~(\ref{e0}) we obtain the desired assertion.
\end{proof}

\begin{example}
Let us construct a mapping which is a weak similarity but not quasisymmetric.  Let $(X,d)$ and $(Y,\rho)$ be semimetric spaces such that $X=Y=[0,\infty)$, $d$ be a Euclidean distance on $[0,\infty)$, $\rho(x,y)=\exp(d(x,y))-1$ and let $f\colon [0,\infty)\to[0,\infty)$ be the identical mapping. It is clear that $(X,d)$ and $(Y,\rho)$ are weakly similar with the realization $(\exp(t)-1,f)$.

Suppose $f$ is $\eta$-quasisymmetric for some fixed $\eta$. Let $x,a,b \in X$ be such that $x=0$ and $a=2b$. It follows from~(\ref{e0}) that
$$
\exp(a)-1\leqslant \eta(2)(\exp(b)-1).
$$
Clearly, this inequality does not hold for sufficiently large $b$. Thus, $f$ is not an $\eta$-quasisymmetric mapping for any fixed function $\eta$.
\end{example}

The following proposition is almost evident.
\begin{proposition}\label{p19}
Let $(X,d)$, $(Y,\rho)$ be semimetric spaces and let $f\colon X \to Y$ be a bijective mapping. Then $f$  is a weak similarity if and only if
the following implications
\begin{equation}\label{ee71}
d(x,y)<d(z,w) \Rightarrow \rho(f(x),f(y))< \rho(f(z),f(w)),
\end{equation}
\begin{equation}\label{ee72}
d(x,y)=d(z,w) \Rightarrow \rho(f(x),f(y))= \rho(f(z),f(w))
\end{equation}
hold for all $x,y,z,w \in X$.
\end{proposition}

The following lemma was formulated and proved in~\cite{PS21U} for the case of metric spaces. In the case of semimetric spaces the proof is the same. But we reproduce it here for the convenience of the reader.
\begin{lemma}\label{l32}
Let $(X,d)$, $(Y,\rho)$ be semimetric spaces and let $f\colon X\to Y$ be an $\eta$-quasisymmetric mapping with $\eta(1)=1$. Then the following implications hold
\begin{equation}\label{e2}
d(x,a)=d(x,b)  \, \Rightarrow \, \rho(f(x),f(a))= \rho(f(x),f(b))
\end{equation}
\begin{equation}\label{e39}
d(x,a)< d(x,b)  \, \Rightarrow \, \rho(f(x),f(a))< \rho(f(x),f(b))
\end{equation}
for all triples $a,b,x$ of points in $X$.
\end{lemma}
\begin{proof}
Let $d(x,a)=d(x,b)$ which is equivalent to
$$
(d(x,a)\leqslant d(x,b))\wedge (d(x,a)\geqslant d(x,b)).
$$
Hence~(\ref{e0}) with $t=1$ implies the right equality in~(\ref{e2}).

Let $d(x,a)< d(x,b)$. Then $d(x,a) = \alpha d(x,b)$ for some $\alpha<1$.
By~(\ref{e0}) we have
\begin{equation}\label{e49}
\rho(f(x),f(a))\leqslant \eta(\alpha)\rho(f(x),f(b)).
\end{equation}
Since $\eta(t)$ is a homeomorphism and $\eta(1)=1$ we have that $\eta(t)$ is strictly increasing and $\eta(\alpha)<1$. Hence, ~(\ref{e49}) implies the right inequality in~(\ref{e39}).
\end{proof}

\begin{theorem}\label{tt3}
Let $(X,d)$, $(Y,\rho)$ be semimetric spaces and let  $f\colon X\to Y$ be a bijective $\eta$-quasisymmetric mapping. If
\begin{equation}\label{e4.7}
\eta(k)\eta\left(\frac 1k \right)=1
\end{equation}
for all $k>0$, then $f$ is a weak similarity.
\end{theorem}

\begin{proof}
Let $x, y, z, w$ be pairwise distinct points from $X$ and let
$$
k_1=\frac{d(x,y)}{d(y,z)}, \quad k_2=\frac{d(z,w)}{d(y,z)}.
$$
Then by~(\ref{e0}),
\begin{equation}\label{e71}
 \rho(f(x),f(y))\leqslant \eta(k_1)\rho(f(y),f(z)),
\end{equation}
\begin{equation}\label{e72}
 \rho(f(z),f(w))\leqslant \eta(k_2)\rho(f(y),f(z))
\end{equation}
and
\begin{equation}\label{e73}
\rho(f(y),f(z))  \leqslant \eta\left(\frac{1}{k_1}\right) \rho(f(x),f(y)),
\end{equation}
\begin{equation}\label{e74}
\rho(f(y),f(z)) \leqslant \eta\left(\frac{1}{k_2}\right) \rho(f(z),f(w)).
\end{equation}

Suppose first $d(x,y)=d(z,w)$, i.e., $k_1=k_2$.  By (\ref{e71}) and (\ref{e74}) we have
\begin{equation}\label{e75}
 \rho(f(x),f(y))\leqslant \eta(k_1)\eta\left(\frac{1}{k_2}\right)\rho(f(z),f(w))
\end{equation}
and by (\ref{e72}) and (\ref{e73}) we have
\begin{equation}\label{e76}
\rho(f(z),f(w)) \leqslant \eta(k_2)\eta\left(\frac{1}{k_1}\right) \rho(f(x),f(y)),
\end{equation}
which by~(\ref{e4.7}) and by the equality $k_1=k_2$ implies
\begin{equation}\label{e77}
\rho(f(x),f(y)) = \rho(f(z),f(w)).
\end{equation}
Suppose now that $d(x,y)<d(z,w)$. Then $k_1<k_2$. This inequality, ~(\ref{e4.7}) and the monotonicity of $\eta$ imply
$\eta(k_1)\eta\left(\frac{1}{k_2}\right)< 1$. The inequality
\begin{equation}\label{e78}
 \rho(f(x),f(y))< \rho(f(z),f(w))
\end{equation}
follows directly from~(\ref{e75}). Thus, implications~(\ref{ee71}) and~(\ref{ee72}) follow and by Proposition~\ref{p19} $f$ is a weak similarity.

Let $a,b,x$ be pairwise distinct points from $X$. It is clear that~(\ref{e4.7}) implies $\eta(1)=1$.
This equality, Lemma~\ref{l32} and Proposition~\ref{p19} complete the proof.
\end{proof}
\begin{remark}
The solution of equation~(\ref{e4.7}) is
$$
\eta(t)=\pm \exp[\Phi(t,1/t)],
$$
where $\Phi(x,z)=-\Phi(z,x)$ is any antisymmetric function of two arguments, see, e.g.~\cite{PM98}.
In particular for $\Phi(x,z)=C(\ln x -\ln z)$ we have that there are solutions of the form $y=\pm x^{2C}$, where $C$ is an arbitrary constant. Note that only those solutions $\eta$ are admissible which are homeomorphisms $\eta\colon [0, \infty)\to [0,\infty)$. Thus, a bijective $\eta$-quasisymmetric mappings are weak similarities for a wide class of functions $\eta$ including all power functions $t^{\alpha}$, $\alpha>0$. It worth to note that in the case $\alpha=1$ by Proposition~\ref{p248} $f$ is a similarity.
\end{remark}

\begin{example}
Consider a mapping $\eta\colon [0, \infty)\to [0,\infty)$ defined as
$$
\eta(t)=
\begin{cases}
(\exp(t)-1)/(\exp(1/t)-1), &t>0, \\
0, &t=0.
\end{cases}
$$
Hence, for $t>0$ we have
$$
\eta'(t)=
\frac{\exp(t)(\exp(1/t)-1)+(1/{t^2})\exp(1/t)(\exp(1/t)-1)}{(\exp(1/t)-1)^2}>0.
$$
It is also clear that $\eta$ is continuous. Thus, $\eta$ is an example of a homeomorphisms  satisfying~(\ref{e4.7}) which is not a power function. In this case
$\Phi(x,z)=\ln(\exp(x)-1)-\ln(\exp(z)-1)$.
\end{example}

\textbf{Concluding remark.} There exist several potentially important applications of the semimetric spaces theory to computer science. In this regard, we only mention the current development of algorithms for solving the Traveling Salesperson Problem. The following
is a citation from the article ``Computer Scientists Break Traveling Salesperson Problem'' by Erica Klarrech:

``Now, in a paper posted online in July, Klein and his advisers at the University of Washington, Anna Karlin and Shayan Oveis Gharan, have finally achieved a goal computer scientists have pursued for nearly half a Century: a better way to find approximate solutions to the traveling salesperson problem.'' (Quanta Magazine, October 2020)

It should be noted here that the preprint~\cite{KKG21} is devoted to the metric TSP problem but the conjecture that the original data set has a metric space structure is not meet in practice in many cases. Therefore, the description of conditions under which the Karlin-Klein-Gharan algorithm has a ``semimetric generalization'' is an important actual problem. Some results in this direction can be found in~\cite{KT21}.



\end{document}